\documentclass[12pt]{amsart}
\usepackage{amsmath,amssymb,amsbsy,amsfonts,latexsym,amsopn,amstext,cite,
                                               amsxtra,euscript,amscd,bm,mathabx,mathrsfs}
\usepackage{url}
\usepackage[colorlinks,linkcolor=blue,anchorcolor=blue,citecolor=blue,backref=page]{hyperref}
\usepackage{color}
\usepackage{graphics,epsfig}
\usepackage{graphicx}
\usepackage{float}
\usepackage[english]{babel}
\usepackage{mathtools}
\usepackage{todonotes}
\usepackage{url}
\usepackage[colorlinks,linkcolor=blue,anchorcolor=blue,citecolor=blue,backref=page]{hyperref}

\usepackage[norefs,nocites]{refcheck}

\hypersetup{breaklinks=true}

\DeclareMathOperator{\tr}{tr}

\def\cA{{\mathcal A}}

\def\cD{{\mathcal D}}

\def\cI{{\mathcal I}}
\def\cJ{{\mathcal J}}

\def\cL{{\mathcal L}}
\def\cM{{\mathcal M}}

\def\cP{{\mathcal P}}

\def\cR{{\mathcal R}}

\def\\{\cr}
\def\({\left(}
\def\){\right)}
\def\fl#1{\left\lfloor#1\right\rfloor}
\def\rf#1{\left\lceil#1\right\rceil}

\newcommand{\bX}{\mathbf{X}}
\newcommand{\bY}{\mathbf{Y}}

\def\rk{\varrho}

\def \C {{\mathbb C}}
\def \F {{\mathbb F}}

\def \K {{\mathbb K}}
\def \N {{\mathbb N}}

\def \Q {{\mathbb Q}}
\def \R {{\mathbb R}}
\def \Z {{\mathbb Z}}

\newcommand{\SL}{\operatorname{SL}}

\newcommand{\Tr}{\operatorname{Tr}}

\newcommand{\rank}{\operatorname{rank}}

\newcommand{\fA}{\mathfrak{A}}
\newcommand{\fB}{\mathfrak{B}}
\newcommand{\fC}{\mathfrak{C}}
\newcommand{\fD}{\mathfrak{D}}
\newcommand{\fE}{\mathfrak{E}}

\newtheorem{thm}{Theorem}

\newtheorem{prop}[thm]{Proposition}

\newtheorem{rem}[thm]{Remark}

\newtheorem{lemma}[thm]{Lemma}

\numberwithin{equation}{section}
\numberwithin{thm}{section}
\numberwithin{table}{section}

\def\mand{\qquad\mbox{and}\qquad}

\begin{document}


\title[Matrices over finite rank multiplicative groups]{Counting matrices over finite rank multiplicative groups}

\author[A. Manning] {Aaron Manning}
\address{School of Mathematics and Statistics, University of New South Wales, Sydney NSW 2052, Australia}
\email{aaron.manning@student.unsw.edu.au}

\author[A. Ostafe] {Alina Ostafe}
\address{School of Mathematics and Statistics, University of New South Wales, Sydney NSW 2052, Australia}
\email{alina.ostafe@unsw.edu.au}

\author[I. E. Shparlinski] {Igor E. Shparlinski}
\address{School of Mathematics and Statistics, University of New South Wales, Sydney NSW 2052, Australia}
\email{igor.shparlinski@unsw.edu.au}

\begin{abstract}
Motivated by recent works on statistics of matrices over sets of number theoretic interest, we study matrices with
entries from arbitrary finite subsets $\cA$ of finite rank multiplicative groups in  fields of characteristic zero.
We obtain upper bounds, in terms of the size of $\cA$, on the number of such matrices of a given rank, with a
given determinant and with a prescribed characteristic polynomial. In particular, in the case of ranks,
our results can be viewed as a statistical version of work by Alon and Solymosi~(2003).
\end{abstract}



\maketitle
\tableofcontents

\section{Introduction}

\subsection{Motivation and set-up}

For a finite subset $\cA$ of a field $\K$, we define $\cM_{m,n}(\cA)$ to be the set of $m \times n$ matrices
with entries from $\cA$. It is also convenient to omit one of the subscripts when $m = n$, writing $\cM_n (\cA)
= \cM_{n,n} (\cA)$. Various counting questions regarding matrices in $\cM_{m,n}(\cA)$ where $\cA$ is a set of
arithmetic significance have been studied in a number of works. Here we are interested in the case when $m$ and $n$
are fixed and the size of $\cA$ grows, that is, when
\[
  A = \# \cA \to \infty.
\]
Thus this is dual to the set-up when $\cA$ is fixed, typically $\cA = \{0, 1\}$ or $\cA = \{- 1, 1\}$, and $m,
n \rightarrow \infty$, which has also received a lot of attention~\cite{BVW, CJMS, Eber, FLM, FJSS, JSS, Kop,
KwSa, NgPa, NgWo, TaVu} as well as the survey~\cite{Vu}.

The direction originates from the case of $\cA = \K = \F_q$, where $\F_q$ is the finite field of $q$
elements~\cite{BreMcK, FeFi, Ful, FuGor, Ger, Rei, Sanna, Wood}.

In characteristic zero, the most studied case is the case of integer entries, bounded (in, say, $\cL^2$ or $\cL^\infty$
norms) by some parameter $H \to \infty$. This direction originates from works of Duke, Rudnick and Sarnak~\cite{DRS},
Eskin, Mozes and Shah~\cite{EMS} and Katznelson~\cite{Kat1,Kat2}, see also~\cite{BlLu, EskKat, GNY, Shah, ShapZhe,
WeXu} for further developments of these techniques, based on geometry of numbers and homogeneous dynamics. More
recently, several different approaches to problems of arithmetic statistics for matrices have emerged~\cite{AhmShp,
ALPS, Afif1, Afif2, AKOS, BlLi, BSW, BOS, BuSh, E-BLS, HOS, MeSh, MOS, OstShp,Shp}.  These works are based on
a variety of other techniques, including some inputs from Diophantine geometry and analytic number theory. In
particular, these new ideas have given the means to approach various counting question for matrices with rational
entries whose numerators and denominators are bounded by a given height $H$, see~\cite{AKOS}, and for matrices
with entries which are polynomial values of integers from $[-H,H]$, see~\cite{BlLi,MOS}, (in the same regime of
fixed $m$ and $n$ and $H \to \infty$).

There is also an emerging direction of studying matrices with entries from a completely general set, where,
surprisingly some nontrivial bounds are possible~\cite{Arut,KPSV,Mud,ShkShp}.

More precisely, most of the above works study the following three subsets of $\cM_{m,n} (\cA)$:
\begin{itemize}
\item  matrices of given determinant $d \in \K$,
\begin{equation}\label{eq:defn:cal-d}
  \cD_{n}(\cA; d) = \{ \bX \in \cM_n (\cA) :~\det\bX = d \},
\end{equation}
\item  matrices with a given characteristic polynomial $f \in \K[T]$,
\begin{equation}\label{eq:defn:cal-p}
  \cP_n (\cA; f) = \{ \bX \in \cM_n (\cA) :~\det(T I_n - \bX) = f \}, 
\end{equation}
\item  matrices of given rank $r \in \N$,
\begin{equation}\label{eq:defn:cal-r}
  \cR_{m,n}(\cA; r) = \{ \bX \in \cM_{m,n} (\cA) :~\rank \:\bX = r \}. 
\end{equation}
\end{itemize}
As with $\cM_n (\cA)$, we also adopt the notation $ \cR_{n}(\cA; r) = \cR_{n, n} (\cA; r)$.

Here we consider the above questions in a new setting, when the set $\cA$ is an arbitrary finite subset of a
multiplicative subgroup $\Gamma$ of finite rank in a field $\K$ of characteristic zero. Besides the aforementioned
works, our motivation also comes from a result of Alon and Solymosi~\cite[Theorem~1]{AlSol}, which shows that
$n \times n$ matrices with entries from finitely generated subgroups $\Gamma$ of $\C^*$ have a rank growing with
their dimension $n$. In our notation, the result of~\cite[Section~4]{AlSol} can be formulated as $\cR_{n}(\cA; r)
= \emptyset$, provided $r < (c_1 \log n)^{c_2}$ for some positive constants $c_1$ and $c_2$, depending only on $n$
and the rank of $\Gamma$.

It is also interesting to note that both the approach of Alon and Solymosi~\cite{AlSol} and our approach are based
on the celebrated {\it Subspace Theorem\/} of Schmidt~\cite{Schm}. More precisely we use its implication for the
number of non-degenerate solutions to linear equations solved over $\Gamma$, given in the currently strongest form
by Amoroso and Viada~\cite[Theorem~6.2]{AmVi}. This in turn has been used in~\cite[Corollary~16]{BGKS} to estimate
the total number of solutions, see Section~\ref{sec:LinEq} for details.

We emphasise that our bounds on the above quantities $\cD_{n}(\cA; d)$ and $\cP_n (\cA; f)$ are uniform with respect to
$d$ and $f$, and the implied constants, also in the bounds on $\cR_{m,n}(\cA; r)$, may only depend on the dimensions
$m$ and $n$ of a matrix, and the rank $\rk$ of $\Gamma$. In Lemmas~\ref{lem:LinEq homog}, \ref{lem:LinEq inhomog},
and~\ref{lem:SystEq} on the number of solutions to linear equations, the implied constant may only depend on the
number of summands $n$ and the rank $\rk$ of $\Gamma$.

\subsection{Notation}

We recall that the notations $U = O(V)$, $U \ll V$ and $ V\gg U$ are equivalent to $|U|\leqslant c V$ for some
positive constant $c$, which, as above, may depend only on $m$, $n$ and $\rk$ throughout this work.

We also write $U \asymp V$ as a shorthand for when both $U \ll V$ and $V \ll U$ hold.

When $S$ is a finite set, we use $\#S$ to denote its cardinality.

Furthermore, throughout, this work we also use
\[
A =\# \cA
\]
to denotes the cardinality of $\cA$.

Finally, $\F_q$ denotes the finite field of $q$ elements and $I_n$ denotes the $n \times n$ identity matrix.

\subsection{Trivial upper bounds}

Before we formulate our results, we record the following trivial bounds, which we use as benchmarks to illustrate
the strength of our results.

Clearly, for any $n \ge 1$ and $\cA \subseteq \K$ of cardinality $A$,
\begin{equation}\label{eq:Triv D}
 \# \cD_{n}(\cA; d )  \ll A^{n^2-1} . 
\end{equation}
In fact, for $\cA = \K = \F_q$ the bound~\eqref{eq:Triv D} is tight.  However, recent work by Shkredov and
Shparlinski~\cite{ShkShp} shows that a better bound is possible for real matrices when $n \geq 3$, without any
further restrictions on the entries.

Also, for $f = T^n + c_{n-1} T^{N-1} + \ldots + c_0\in \K[T]$, 
\begin{equation}\label{eq:Triv P}
 \#  \cP_n (\cA; f) \ll A^{n^2-2} . 
\end{equation}
Indeed,  writing $\bX = (x_{i,j})_{i,j=1}^n$ and using $f = \det(TI_n - \bX)$ we see that $\bX$ has a fixed trace
$\Tr \bX = - c_{n-1}$, and hence we can express $x_{n,n}$ via other diagonal elements. After this the equation $
\det \bX = (-1)^n c_0$ becomes an algebraic equation in $n^2-1$ variables. This equation is nontrivial as one can
see by specialising all non-diagonal elements of $\bX$ to $0$.

Furthermore, for any $n \geq m \geq r\ge 1$ and $\cA \subseteq \K$ of finite cardinality $ A$,
\begin{equation}\label{eq:Triv R}
    \#\cR_{m,n}(\cA; r) \ll A^{nr + mr - r^2},
\end{equation}
Indeed, without loss of generality, we can assume that the top left $r \times r$ submatrix of $\bX \in \cM_{m,n}
(\cA)$ is non-singular. Then we see that after fixing $nr$ elements in the top $r$ rows of $\bX$ and, the $(m-r)
r$ remaining elements in the first $r$ columns of $\bX$, the remaining elements are uniquely defined.

\section{Main Results}

\subsection{Matrices of given rank}

Recall the definition of $\cR_{m,n} (\cA; r)$ given in~\eqref{eq:defn:cal-r} as the number of $m \times n$ matrices
over $\cA$ of rank $r$.

\begin{thm} \label{thm:rank}
  Let $\K$ be a field of characteristic zero, and $\Gamma$ a rank $\rk $ subgroup of ${\K}^{\ast}$. If $n,m \geq 2$
  with $n \geq m \geq r > 0$, for any finite subset $\cA$ of $\Gamma$ with cardinality $A$, we have
  \[
    \#\cR_{m,n}(\cA; r) \ll \begin{cases}
      A^{nr + m - r}, \hfill \quad & 2m \leq n + r, \\
      A^{nr + m - r + \fl{\frac{r - 1}{2}} (2m - n - r)}, \quad & \text{otherwise}. 
    \end{cases}
  \]
\end{thm}

When $2m > n + r$,  Theorem~\ref{thm:rank} gives us a saving against the trivial bound~\eqref{eq:Triv R}, of
\[
  \frac{A^{nr + mr - r^{2}}}{A^{nr + m - r + \fl{\frac{r - 1}{2}} (2m - n - r)}} = \begin{cases}
    A^{\frac{1}{2} (n - r)(r - 1)}, \quad & r\ \text{odd}, \\
    A^{\frac{1}{2} r(n - r) + (m - n)}, \quad & r\ \text{even}.
  \end{cases}
\]   

When $2m \leq n + r$, the bound of Theorem~\ref{thm:rank} is tight. For instance, take $\K = \Q$, $\Gamma = \langle
2\rangle$ and $\cA_k = \{ 2^s :~1 \leq s \leq 2k\}$, defining $A_k = \#\cA_k = 2k$. In this case we have $k^{nr}
\asymp A_{k}^{nr}$ ways of fixing the first $r$ rows with elements of the form $2^{s}$ for $1 \leq s \leq k$. We
then have $k^{m - r} \asymp A_k^{m - r}$ ways of choosing all the other rows to be $2^s$ multiplied by the first
row for each $1 \leq s \leq k$ (to guarantee elements stay within $\cA_k$). Thus the number of matrices of rank
at most $r$ satisfies
\begin{equation} \label{eq: rank LB}
  \sum_{j = 1}^r \#\cR_{m,n}(\cA_{k}; j) \asymp A_k^{nr + m - r}.
\end{equation}
Therefore
\begin{align*}
  \#\cR_{m,n}(\cA_{k}; r) 
  &= \sum_{j = 1}^r \#\cR_{m,n}(\cA_k; j) - \sum_{j = 1}^{r - 1} \#\cR_{m,n}(\cA_k; j) \\
  &\gg  \sum_{j = 1}^r \#\cR_{m,n}(\cA_k; j)  + O\( A_k^{n(r - 1) + m - (r - 1)}\) \\
  &\gg A_{k}^{nr + m - r}.
\end{align*}
 
\subsection{Matrices of given determinant}

We recall the definition of $\cD_n (\cA; d)$ given in~\eqref{eq:defn:cal-d} as the number of $n \times n$ matrices
over $\cA$ with determinant $d$.

\begin{thm} \label{thm:det}
  Let $\K$ be a field of characteristic zero, and $\Gamma$ a rank $\rk $ subgroup of $\K^\ast$. For any $d \in
  {\K}$, and finite subset $\cA$ of $\Gamma$ with cardinality $A$,
  \[
    \#\cD_{n}(\cA; d) \ll \begin{cases}
      A^{n^{2} - \left\lceil \frac{n}{2} \right\rceil}, \quad & d = 0, \\
      A^{n^{2} - \left\lceil \frac{n + 1}{2} \right\rceil}, \quad & d \neq 0.
    \end{cases}
  \]
\end{thm}

Clearly Theorem~\ref{thm:det} always improves the bound~\eqref{eq:Triv D} (and also the stronger bound
from~\cite{ShkShp}), except when $n=2$ and $d=0$, in which case the bound is tight.  Indeed, specialising the
lower bound~\eqref{eq: rank LB} to the case when $m = n$ and $r = n - 1$, we immediately see that
\[
  \#\cD_n (\cA; 0) \gg A^{n^2 - n + 1}.
\]
Hence for $d = 0$, Theorem~\ref{thm:det} is tight when $n = 2$ and $n = 3$.

\subsection{Matrices of given characteristic polynomial}

Recall the definition of $\cP_n(\cA; f)$ given in~\eqref{eq:defn:cal-p} as the number of $n \times n$ matrices
over $\cA$ with characteristic polynomial $f$.

We consider the case $n = 2$ first separately, due to the compatibility of the formulae for the trace and determinant
in this case, which allows for a tighter bound to be acquired than in the general case.

\begin{thm} \label{thm:poly 2-by-2}
  Let $\K$ be a field of characteristic zero, and $\Gamma$ a rank $\rk $ subgroup of $\K^\ast$. For any $d,t \in
  \K$ not both zero, and finite subset $\cA$ of $\Gamma$ with cardinality $A$,
  \[
    \#\cP_2 (\cA; T^2 - tT + d) \ll \begin{cases}
      A, \quad & dt = 0, \\
      1, \quad & dt \neq 0.
    \end{cases}
  \]
\end{thm}

We do not specify a bound in Theorem~\ref{thm:poly 2-by-2} when $d = t = 0$ because in this case, the trivial bound
of $O(A^2)$ is tight. This can be seen with the following construction. Let $\cA_{k} = \{ \pm 2^s :~0 \leq s <
k \}$, with $A_k = \#\cA_k = k$.  A matrix $\bX = (x_{i,j})_{i,j = 1}^2$ is in $\cP_2 (\cA_k; T^2)$ if and only if
\[
  \det\bX = x_{1,1}x_{2,2} - x_{1,2}x_{2,1} = 0 \mand \tr \bX = x_{1,1} + x_{2,2} = 0,
\]
or equivalently,
\[
  x_{1,1}^2 = - x_{1,2}x_{2,1} \mand x_{1,1} = - x_{2,2}.
\]

By writing $x_{1,1} = 2^{a+b}$, $x_{1,2} = 2^{2a}$ and $x_{2,1} = - 2^{2b}$ with non-negative integers $a,b <
k/2$, we see that
\[
  \cP_2 (\cA_k; T^2) \gg A_k^2.
\]

We now turn our attention to the case $n \geq 3$. Our bound is based on fixing only the coefficients of $T^{n - 1}$
and $T^{n - 2}$ in the characteristic polynomial $f \in \K[T]$, as motivated by the techniques of~\cite{AKOS}. While we
do not present a matching lower bound, the strength of~\cite[Theorem~2.3]{AKOS} which uses this approach indicates
its unexpected power.

Before stating Theorem~\ref{thm:poly general} for fixed characteristic polynomial, we introduce the 
following function
\begin{equation}
  \begin{split} \label{eq:alpha_n}
    \alpha(n) = \frac{n(n - 1)}{2}
      &+ \max\Biggl\{
        \fl{\frac{n - 1}{2}} + \fl{\frac{n(n - 1)}{4}},\\
      & \qquad \qquad  \qquad \qquad\fl{\frac{n}{2}}  + \fl{\frac{n(n - 1)}{4} - \frac{1}{2}}
      \Biggr\}.
 \end{split} 
\end{equation}

\begin{thm} \label{thm:poly general}
  Let $\K$ be a field of characteristic zero, and $\Gamma$ a rank $\rk $ subgroup of $\K^\ast$. For any monic
  polynomial $f$ of degree $n \geq 3$, and $\cA \subseteq \Gamma$ of finite cardinality $A$, we have
  \[
    \#\cP_n(\cA; f) \ll A^{\alpha(n)}
  \]
  where $\alpha(n)$ is given by~\eqref{eq:alpha_n}. 
\end{thm}


We note that 
\[
  \lim_{n\to \infty} \alpha(n)/n^2 = 3/4.
\]

Direct calculations show that Theorem~\ref{thm:poly general} also improves~\eqref{eq:Triv P} and also the bound
\[
  \#\cP_n(\cA; f) \ll \#\cD_n(\cA; (-1)^n c_0 ) \ll 
  \begin{cases}
      A^{n^{2} - \left\lceil \frac{n}{2} \right\rceil}, \quad & f(0)= 0, \\
      A^{n^{2} - \left\lceil \frac{n + 1}{2} \right\rceil}, \quad & f(0) \neq 0,
    \end{cases}
\]
which follows from Theorem~\ref{thm:det}.

\begin{rem} We note that in the proof of Theorem~\ref{thm:poly general} we derive more precise bounds which
depend on some the properties of the coefficients of $X^{n-1}$ and $X^{n-2}$ from $f$,
see Appendix~\ref{app: bound poly general}, where these bounds are presented.
\end{rem} 

Finally, we observe that  Theorems~\ref{thm:poly 2-by-2} and~\ref{thm:poly general} imply upper bounds on the
number {\it cyclotomic\/} $\bX \in  \cM_n (\cA)$, that is, matrices with $\bX^k = I_n$ for some positive integer $k$.

\section{Linear equations in finite rank multiplicative groups} 
 \label{sec:LinEq} 

\subsection{Counting non-degenerate solutions}

We start with the best known bound in the case of arbitrarily many summands in an arbitrary field of characteristic
zero due to Amoroso and Viada~\cite[Theorem~6.2]{AmVi}, however as in~\cite{BGKS} the previous bound of 
Evertse, Schlickewei and Schmidt~\cite{EvSchSch} is also suitable for our purpose (as well as other bounds 
of this kind).

Let $\K$ be a field of characteristic zero, and let $\Pi$  be a subgroup of $(\K^\ast)^n$.  We say that a solution
to the equation
\begin{equation} \label{eq:LinEq Pi}
  a_1 x_1 + \ldots + a_n x_n = 1, \quad (x_1, \ldots, x_n) \in \Pi.
\end{equation}
is non-degenerate if
\[
  \sum_{i \in \cI} a_i x_i \neq 0
\]
for all $\cI \subseteq \{1, 2, \ldots, n\}$.

\begin{lemma} \label{lem:LinEq-nondegen}
  Let $\K$ be a field of characteristic zero, and $\Pi$ a rank $\rk $ subgroup of $(\K^\ast)^n$. For any $a_1,
  \ldots, a_n \in \K^\ast$, the number of non-degenerate solutions to~\eqref{eq:LinEq Pi} is at most
  $(8n)^{4n^4 (n + \rk + 1)}$.
\end{lemma}

\subsection{Counting arbitrary solutions}
Since the entries of our matrices are drawn from a subgroup $\Gamma$ of $\K^\ast$, we specialise
Lemma~\ref{lem:LinEq-nondegen} to the case $\Pi = \Gamma^n$.

The following result is essentially~\cite[Corollary~16]{BGKS}. Although it is presented in~\cite{BGKS} for $\K =
\C$ with integer coefficients, it extends to arbitrary fields of characteristic zero in the natural way.

\begin{lemma} \label{lem:LinEq homog}
  Let $\K$ be a field of characteristic zero, and $\Gamma$ a rank $\rk $ subgroup of $\K^\ast$. Suppose that $\cA
  \subseteq \Gamma$ is a finite set of cardinality $A$.  For any  $a_1, \ldots, a_n \in \K^\ast$, the number of
  solutions to
  \[
    a_1 x_1 + \ldots + a_n x_n = 0,\quad x_1, \ldots, x_n \in \cA,
  \]
  is $O\(A^{\fl{n/2}} \)$. 
\end{lemma}

It is easy to see that the bound of Lemma~\ref{lem:LinEq homog} is tight, since for any choice of $\Gamma$ and
$\cA$, if $n = 2k$ then we can choose
\[
  a_{1} = \ldots = a_{k} = 1, \mand a_{k + 1} = \ldots = a_{2k} = - 1,
\]
allowing us to construct $A^{k} = A^{\fl{n/2}}$ solutions by setting $x_{i} = x_{k + i}$ for all
$i \in \left\{ 1,\ldots,k \right\}$. If $n = 2k + 1$, then we may similarly consider
\[
  a_1 = \ldots = a_{k - 1} = 1, \quad a_k = \ldots = a_{2k} = - 1, \mand a_{2k + 1} = 2,
\]
which allows us to once again construct $A^k =A^{\fl{n/2}}$ solutions by setting $x_i = x_{k + i - 1}$
for all $i \in \{1, \ldots, k - 1\}$ and $x_{2k - 1} = x_{2k} = x_{2k + 1}$.

For problems such as counting matrices of a given non-zero determinant, we also require a non-homogeneous (and a
slightly stronger) version of Lemma~\ref{lem:LinEq homog} where the right hand side of the corresponding equation
is an arbitrary $a_0 \in \K^\ast$. We derive it as an application of Lemma~\ref{lem:LinEq homog}, which we use to
handle the vanishing subsums present in degenerate solutions.

\begin{lemma} \label{lem:LinEq inhomog}
  Let $\K$ be a field of characteristic zero, and $\Gamma$ a rank $\rk $ subgroup of $\K^\ast$. Suppose that $\cA
  \subseteq \Gamma$ is a finite set of cardinality $A$.  For any  $a_0, a_1, \ldots, a_n \in \K^\ast$, the number
  of solutions to
  \[
    a_1 x_1 + \ldots + a_n x_n = a_0,\quad x_1, \ldots, x_n \in \cA,
  \]
  is $O\(A^{\fl{(n-1)/2}} \)$. 
\end{lemma}

\begin{proof}
  Dividing all coefficients of the above equation by $a_0$ we see that it is sufficient to consider the equation
  \begin{equation} \label{eq:LinEq =1}
    a_1 x_1 + \ldots + a_n x_n = 1,\quad x_1, \ldots, x_n \in   \cA.
  \end{equation}

  Let $\fA$ denote the number of solutions to~\eqref{eq:LinEq =1}, and for each such solution
  $\boldsymbol{x} = (x_1, \ldots, x_n)$, associate a subset $\cI(\boldsymbol{x}) \subseteq \{1, \ldots, n\}$
  with the largest cardinality such that
  \[
    \sum_{i \in \cI(\boldsymbol{x})} a_i x_i = 0.
  \]

  For each $\cI \subseteq \{1, \ldots, n\}$, let $\fA_\cI$ denote the number of solutions of~\eqref{eq:LinEq =1}
  such that $\cI(\boldsymbol{x}) = \cI$. As such, there is a particular set $\cJ$ which maximises the number of
  corresponding solutions such that
  \[
         \fA
    =    \sum_{\cI \subsetneq \{1, \ldots, n\}} \fA_\cI
    \ll  \fA_\cJ.
  \]

  Considering now just solutions $\boldsymbol{y}$ for which $\cI(\boldsymbol{y}) = \cJ$ in the interest of bounding
  $\fA_\cJ$, we may split~\eqref{eq:LinEq =1} into the maximal degenerate part
  \begin{equation} \label{eq: Degen part}
    \sum_{i \in \cJ} a_i x_i = 0
  \end{equation}
  and non-degenerate part
  \begin{equation} \label{eq: Nondegen part}
    \sum_{i \in \left\{ 1,\ldots,n \right\} \setminus \cJ} a_{i}x_{i} = 1.
  \end{equation}

  Because solutions to~\eqref{eq: Nondegen part} are non-degenerate by construction, the number of solutions is
  $\fB \ll 1$ by Lemma~\ref{lem:LinEq-nondegen} with $\Pi = \Gamma^{n - \#\cJ}$.

  By construction, $\#\cJ \leq n - 1$ and hence the number of solutions $\fC$ to~\eqref{eq: Degen part} satisfies
  \[
    \fC \ll A^{\fl{(n-1)/2}} 
  \]
  by Lemma~\ref{lem:LinEq homog} (except at $n = 1$, in which case the theorem we presently prove is trivial). This
  leads to the overall bound
  \[
    \fA \ll \fA_\cJ \ll \fB\fC \ll A^{\fl{(n-1) /2}}, 
  \]
  concluding the proof.
\end{proof}

As we saw when illustrating the tightness of Lemma~\ref{lem:LinEq homog}, for the appropriate choice of $a_1,
\ldots, a_{n - 1}$ we have $A^{\fl{(n-1) /2}}$ solutions to
\[
  a_1 x_1 + \ldots + a_{n - 1} x_{n - 1} = 0.
\]
Choosing now $a_n =  1$ and $a_0 = x_n$ for some fixed $x_n \in \cA$, we see that Lemma~\ref{lem:LinEq inhomog}
is also tight.

We also require a bound on the number of solutions to a rather special system of two equations with elements
of $\Gamma$.

\begin{lemma} \label{lem:SystEq}
  Let $\K$ be a field of characteristic zero, and $\Gamma$ a rank $\rk $ subgroup of $\K^\ast$. Suppose that $\cA
  \subseteq \Gamma$ is a finite set of cardinality $A$.  The number of solutions to the system of equations
  \begin{equation} \label{eq:SystEq}
    x_1 + \ldots + x_n = x_1^2 + \ldots + x_n^2 = 0, \quad x_1, \ldots, x_n \in \cA,
  \end{equation}
  is $O\(A^{\fl{2n/5}}\)$.
\end{lemma}

\begin{proof}
For each partition 
\[
  \{1, \ldots, n \} = \bigsqcup_{i=1}^h \cI_i
\]
into $h\ge 1$ disjoint sets $\cI_j$, with $\# \cI_j\ge 2$, $j =1, \ldots, h$, we count solutions to
$x_1 + \ldots + x_n=0$, which form non-degenerate solutions to each of the equations
\[
\sum_{i \in \cI_j} x_i = 0, \quad j =1, \ldots, h.
\]
Fixing one term of each equation and counting solutions in the remainder using Lemma~\ref{lem:LinEq-nondegen},
there are $O(A^h)$ such solutions.

Let $k$ be the number of sets $\cI_j$ with $\# \cI_j = 2$, where, without loss of generality we can assume that
\[
\cI_j = \{2j-1, 2j\}, \quad j =1, \ldots, k.
\]
Hence $h \le k + \fl{(n-2k)/3}$ and thus there are at most 
\begin{equation} \label{eq: Bound 1}
  T_1 \ll A^{k + \fl{(n-2k)/3}} =  A^{\fl{(n+k)/3}}
\end{equation}
such solutions. 

On the other hand since we now have $x_{2j} = - x_{2j-1}$ for $j \leq k$, the equation
$x_1^2 + \ldots + x_n^2 = 0$ becomes 
\[
  2 \sum_{j=1}^k x_{2j-1}^2 + \sum_{j=2k+1}^n x_j^2 = 0, 
\]
which by Lemma~\ref{lem:LinEq homog} has at most 
\begin{equation} \label{eq: Bound 2}
  T_2 \ll A^{\fl{(n-k)/2}}
\end{equation}
solutions, after which the remaining variable $x_{2j}$, $j=1, \ldots, k$, are uniquely defined. 

Choosing, for each $k \in \{0, \ldots, \fl{n/2}\}$, one of the bounds~\eqref{eq: Bound 1} or~\eqref{eq: Bound 2},
whatever is smaller, we deduce that the number of solutions to~\eqref{eq:SystEq} is $O(A^{\kappa_n})$ where
\[
  \kappa_n = \max_{k \in \{0, \ldots, \fl{n/2}\}}  \min\left\{ \fl{(n+k)/3}, \fl{(n-k)/2}\right\}.
\]

By noticing that if $k \leq n / 5$ then $(n + k) / 3 \leq 2n / 5$, and similarly that if $k \geq n / 5$ then
$(n - k) / 2 \leq 2n / 5$, it follows that
\[
  \kappa_n \leq \frac{2n}{5},
\]
and by the integrality of $\kappa_n$ the result follows. 
\end{proof}

\begin{rem}
  It is not difficult to further show that $\kappa_n = \fl{2n/5}$ in the proof of Lemma~\ref{lem:SystEq}, with
  the maximum attained at $k = \fl{n/5}$.
\end{rem} 

%

\section{Proofs of main results} 

\subsection{Proof of Theorem~\ref{thm:rank}}

Our proof employs several ideas introduced in~\cite[Theorem~2.1]{MOS}, with the appropriate alterations made to
use Lemmas~\ref{lem:LinEq homog} and~\ref{lem:LinEq inhomog}, which are the new tools available in our setting.

As in the derivation of~\eqref{eq:Triv R}, we simplify by counting the size of the set $\cR_{m,n}^\ast(\cA;r)$
of matrices in $\cR_{m,n}(\cA;r)$ in which the top left $r \times r$ submatrix is non-singular.

For arbitrary $\bX = (x_{i,j})_{i,j = 1}^n \in \cR_{m,n}^\ast(\cA; r)$, we may write $\bX$ as the block matrix
\[
  \bX = \begin{bmatrix}
    \bX_1 & \bX_2 \\
    \bX_3 & \bX_4
  \end{bmatrix},
\]
where
\[
  \bX_1 = (x_{i,j})_{i, j = 1}^r
\]
is the $r \times r$ non-singular submatrix which exists by assumption.

There are at most \[\fA \ll A^{r^2}\] possible values for the entries of $\bX_1$.

Observe that for each integer $k \in \{r + 1, \ldots, m\}$, the $k$-th row of $\bX$ is a unique linear combination of
the first $r$ rows, given by coefficients $\rho_1 (k), \ldots, \rho_r (k) \in \K$. We say that $\bX_3$, the matrix
immediately below $\bX_1$, is of type $t \in \{1, \ldots, r\}$ if $t$ is the largest number of non-zero values
among the coefficients $\rho_1 (k),\ldots,\rho_r (k)$ taken over each $k \in \{r + 1, \ldots, m\}$. Suppose that,
in particular the $h$-th row is such that $t$ of the coefficients are non-zero, that is, $h$ corresponds with the
row which maximises the value of $t$. Without loss of generality we assume that it is the first $t$ coefficients
$\rho_1 (h), \ldots, \rho_t (h)$ which are non-zero. It is therefore possible to choose a non-singular $t \times t$
submatrix of
\[
  (x_{i,j})_{1 \leq i \leq t, 1 \leq j \leq r},
\]
which we assume, without loss of generality, to be
\[
  (x_{i,j})_{i, j = 1}^t.
\]

This means that each of the $A^t$ choices for $(x_{h,1}, \ldots, x_{h,t})$ defines fully the coefficients $(\rho_1
(h), \ldots, \rho_t (h))$ and subsequently $\left( \rho_1 (h),\ldots,\rho_r (h) \right)$ by including the zero
values. Thus the values of $x_{h,j}$ for $j \in \{t + 1, \ldots, r\}$ are also fixed, that is, the rest of the
corresponding row of $\bX_3$. Since $h$ has been chosen to maximise the value of $t$, we can apply the same bound to
each row of $\bX_3$ to deduce that for each $\bX_1$ there are
\[
  \fB_t = \prod_{j = r + 1}^m A^t \ll A^{t(m - r)}
\]
corresponding possible matrices $\bX_3$ of type $t$.

Given such an $h$ as described above, for each column indexed by $j \in \{r + 1, \ldots, n\}$ we have an equation
\begin{equation} \label{eq: linear comb}
  \rho_1 (h) x_{1,j} + \ldots + \rho_r (h) x_{r,j} = x_{h,j}
\end{equation}
determining the value of $x_{h,j}$ in terms of the value in the $j$-th column of the first $r$ rows.

Solving this equation in $(x_{1,j}, \ldots, x_{r,j}, x_{h,j})$ for each $j$ as described above fixes the upper
right $r \times (n - r)$ submatrix, along with the remainder of the $h$-th row. This means that the analogous equation
\[
  \rho_1 (i)x_{1,j} + \ldots + \rho_r (i)x_{r,j} = x_{i,j}
\]
for each $i \in \{r + 1, \ldots, m\} \setminus \{h\}$, with potentially fewer non-zero coefficients, has a
fixed left hand side, and thus $x_{i,j}$ on the right hand side is uniquely determined.

Let $\fC_t$ be the maximum number of solutions to the equation~\eqref{eq: linear comb} in $(x_{1,j}, \ldots,
x_{r,j}, x_{h,j})$ for each $j \in \{r + 1, \ldots, n\}$. Given we require $n - r$ such equations for each $j$
to count all remaining values of $\bX$, summing over all possibles types $t$, we have an overall bound of
\begin{equation} \label{eq:rank:preliminary-bound}
      \#\cR_{m,n}(\cA; r)
  \ll \#\cR_{m,n}^\ast(\cA; r)
  \ll \fA\sum_{t = 1}^{r}\fB_t \fC_t^{n - r}.
\end{equation}

Subtracting $x_{h,j}$ from both sides of~\eqref{eq: linear comb}, we have an equation of the same form as in
Lemma~\ref{lem:LinEq homog} with $t + 1$ non-zero coefficients, and so we have
\[
  \fC_t \ll A^{\fl{\frac{t + 1}{2}} + r - t},
\]
where the factor of $A^{r   -  t}  = A^{(r + 1) - (t + 1)}$ counts the number of solutions in the ``free'' variables
corresponding to the zero coefficients.

Now, computing the bound in~\eqref{eq:rank:preliminary-bound} we have
\begin{align*}
  \#\cR_{m,n}(\cA; r)
    &\ll \fA\sum_{t = 1}^t \fB_t \fC_t^{n - r} \\
    &\ll A^{r^2}\sum_{t = 1}^r A^{t(m - r)}\left( A^{\fl{\frac{t + 1}{2}} + r - t} \right)^{n - r} \\
    &\ll \sum_{t = 1}^r A^{r^2 + t(m - r) + \fl{\frac{t + 1}{2}} (n - r) + (r - t)(n - r)} \\
    &\ll \max\limits_{t \in \{1, \ldots, r\}}A^{r^2 + t(m - r) + \fl{\frac{t + 1}{2}} (n - r) + (r - t)(n - r)}.
\end{align*}

By defining
\[
  \delta(n,m,r,t) = r^2 + t(m - r) + \fl{\frac{t + 1}{2}} (n - r) + (r - t)(n - r),
\]
we may write
\begin{equation} \label{eq:rank:bound-in-delta}
    \#\cR_{m,n}(\cA; r) \ll \max\limits_{t \in \{1, \ldots, r\}}  A^{\delta(n,m,r,t)}.
\end{equation}

Simplifying, we find that
\begin{align*}
     \delta(n,m,r,t)
  &= mt + \fl{\frac{t + 1}{2}} (n - r) - nt + nr \\
  &= nr + t\left( m - \frac{n + r}{2} \right) + \begin{cases}
    \frac{n - r}{2},\quad & t\text{ odd}, \\
    0,\quad & t\text{ even}.
  \end{cases}
\end{align*}

If $2m \leq n + r$, then the maximum value of $\delta$ over $t$ corresponds to 
\begin{equation}
\label{eq: t small m}
t = 1. 
\end{equation}

If $2m > n + r$, then  $\delta(n,m,r,t)$ is strictly monotonically increasing over integers $t$ of the same
parity. Thus it suffices to check the two possibilities $t \in \{r, r - 1\}$. As such, we consider that
\begin{align*}
  \delta(n,m,r,r) & - \delta(n,m,r,r - 1) \\
   & = r\left( m - \frac{n + r}{2} \right) - (r - 1)\left( m - \frac{n + r}{2} \right) + (- 1)^{r + 1}\frac{n - r}{2} \\
   & = \left( m - \frac{n + r}{2} \right) + (- 1)^{r + 1}\frac{n - r}{2} \\
   & = \begin{cases}
    m - r, \quad & r\ \text{odd}, \\
    m - n, \quad & r\ \text{even}.
  \end{cases}
\end{align*}

In particular, for odd $r$, we have $\delta(n,m,r,r) \geq \delta(n,m,r,r - 1)$, while for even $r$, we have
$\delta(n,m,r,r) \leq \delta(n,m,r,r - 1)$. Therefore the choice of $t$ which maximises $\delta$ is given by
\[
  t =\begin{cases}
    r, \quad & \text{if $r$ is odd}, \\
   r-1,\quad &\ \text{if $r$ is even}, 
  \end{cases}
 \]
 or equivalently
\begin{equation}
\label{eq: t large m}
t = 2\fl{\frac{r - 1}{2}} + 1.
\end{equation}

Therefore, in the case when when $2m\le n+r$, with the choice of $t$ in~\eqref{eq: t small m}, 
\[
  \delta(n,m,r,t)
 = nr + \left( m - \frac{n + r}{2} \right) + \frac{n - r}{2}   = nr + m - r.
\]
while for  $2m>n+r$,  with the choice of $t$ in~\eqref{eq: t large m}, 
we have
\begin{align*}
  \delta(n,m,r,t)
    & = nr + \left( 2\fl{\frac{r - 1}{2}} + 1 \right)\left( m - \frac{n + r}{2} \right) + \frac{n - r}{2} \\
    & = nr + m - r + \fl{\frac{r - 1}{2}} (2m - n - r).
\end{align*}

Substituting these into~\eqref{eq:rank:bound-in-delta}, we conclude the proof.

\subsection{Proof of Theorem~\ref{thm:det}}

For the case when $d = 0$, we may write $\cD_n (\cA; 0)$ as the set of matrices which have rank strictly
less than $n$. Therefore,
\[
  \#\cD_n (\cA; 0) = \sum_{r = 1}^{n - 1} \#\cR_{n} (\cA; r).
\]

Applying now Theorem~\ref{thm:rank} (when $2m > n + r$, which in our case $m=n$ is equivalent to $r < n$), we deduce
\begin{equation}
  \begin{split} 
      \#\cD_n(\cA; 0)
      &\ll \sum_{r = 1}^{n - 1}A^{nr + n - r + \fl{\frac{r - 1}{2}} (n - r)}  \\
      &\ll \max\limits_{r \in \{1, \ldots, n - 1\}} A^{nr + n - r + \fl{\frac{r - 1}{2}} (n - r)}. \label{eq:sum-of-ranks}
   \end{split} 
\end{equation}

Defining
\[
  \delta(n,r) = nr + n - r + \fl{\frac{r - 1}{2}} (n - r)
\]
for the exponent in the above expression, we have
\begin{align*} 
  \delta(n,r) &= nr + n - r + \frac{r}{2}(n - r) + (r - n) \cdot \begin{cases}
      \frac{1}{2}, \quad & r\ \text{odd}, \\
      1, \quad & r\ \text{even},
    \end{cases} \\
  &= r\left( \frac{3}{2}n - \frac{r}{2} - 1 \right) + n + (r - n) \cdot \begin{cases}
    \frac{1}{2}, \quad & r\ \text{odd}, \\
    1, \quad & r\ \text{even}.
  \end{cases}
\end{align*}
Straightforward computations similar to those in the proof of Theorem~\ref{thm:rank} show that $\delta(n, r)$ is
increasing over integers $r$ of the same parity, and that as a function of $r \in \{1, \ldots, n - 1\}$,
it is maximised at $r = n - 1$. Substituting this back in~\eqref{eq:sum-of-ranks} we deduce
\begin{equation}
  \#\cD_n (\cA; 0)  \ll A^{n^{2} - \rf{\frac{n}{2}}},
     \label{eq:zero-conclusion}
\end{equation}
proving the case $d = 0$.




Suppose now that $d \neq 0$. Take a matrix $\bX \in \cD_n (\cA; d)$ and consider the Laplace expansion for the
determinant across the first row given by
\begin{equation} \label{eq:cofactor}
  \det\bX = d = \sum_{j = 1}^n (- 1)^{j + 1} x_{1,j} \det\bX_{1,j},
\end{equation}
where $\bX_{1,j}$ is the submatrix of $\bX$ obtained by removing the first row and $j$-th column.

Suppose firstly that none of the minors $\det\bX_{1,j}$ in~\eqref{eq:cofactor} are zero. In this case, we have at most
\[
  \fA \ll A^{n^2 - n}
\]
possibilities for the bottom $n - 1$ rows of $\bX$. Given that none of the minors are zero, we can count the number
of solutions to~\eqref{eq:cofactor} in the variables $x_{1,j}$ in
\[
  \fB \ll A^{\fl{\frac{n - 1}{2}}}
\]
ways by Lemma~\ref{lem:LinEq inhomog}, leading to an overall bound of
\begin{equation} \label{eq: det Xij nonvanish} 
  \fA\fB \ll A^{n^{2} - n + \fl{\frac{n - 1}{2}}}   = A^{n^2 - \rf{\frac{n + 1}{2}}}. 
\end{equation}

Now, suppose that at least one of the minors $\det\bX_{1,j}$ is zero, which now excludes the possibility $n = 2$.
We can assume, without loss of generality, that in particular, $\det\bX_{1,1} = 0$. Thus, there are at most
\[
  \fC = \#\cD_{n - 1} (\cA; 0)
  \ll A^{(n - 1)^2 - \rf{\frac{n - 1}{2}}}
  = A^{n^{2} - 2n + 1 - \rf{\frac{n - 1}{2}}}
\]
possibilities for $\bX_{1,1}$ by~\eqref{eq:zero-conclusion}.

We can then fix the elements $x_{1,1},x_{2,1},\ldots,x_{n,1}$ in \[\fD = A^{n}\] ways, leaving only the first row
less the top left entry unfixed. Under these assumptions,~\eqref{eq:cofactor} becomes
\begin{equation} \label{eq:cofactor-reduced}
  d = \sum_{j = 2}^n (- 1)^{j + 1}x_{1,j} \det\bX_{1,j}.
\end{equation}

Let $\fE_t$ be the number of solutions to~\eqref{eq:cofactor-reduced} under the assumption that exactly $t$ of
the matrix minors are non-zero. We assume, without loss of generality, that it is the first $t$ coefficients of
the variables  $x_{1,2}, \ldots, x_{1, (t + 1)}$ which are non-zero.

If $t=1$ then
\[
  d = - x_{1,2}\det\bX_{1,2},
\]
where $\det\bX_{1,2} \ne 0$ is already fixed. This defines $x_{1,2}$ uniquely, while the remaining variables can
be fixed in $A^{n - 2}$ ways leading to a bound of
\[
  \fE_1 \ll A^{n - 2}.
\]

If $t = 2$, then we have an equation
\[
  d = - x_{1,2}\det\bX_{1,2} + x_{1,3}\det\bX_{1,3}
\]
with $O(1)$ solutions by Lemma~\ref{lem:LinEq inhomog}, while the remaining elements can be fixed in $A^{n - 3}$
ways leading to a bound of
\[
  \fE_2 \ll A^{n - 3}.
\]

If $3 \leq t \leq n - 1$, which may only happen when $n \geq 4$, we can solve for the non-zero coefficients in
$A^{\fl{\frac{t - 1}{2}}}$ by Lemma~\ref{lem:LinEq inhomog} ways and the remaining coefficients in $A^{n - 1 - t}$
ways leading to \[\fE_{t} \ll A^{n - 1 - t + \fl{\frac{t - 1}{2}}}.\] We observe that for $t=1$ this also formally
coincides with the above bound on $\fE_1$.

Combining these, we have a total bound on the number of matrices in $\cD_n(\cA; d)$ which have a singular submatrix
in the Laplace expansion as
\begin{align*}
  \fC\fD\fE_t & \ll A^{n^2 - 2n + 1 - \rf{\frac{n - 1}{2}}} \cdot A^{n} \cdot \begin{cases}
    A^{n - 3}, & t = 2, \\
    A^{n - 1 - t + \fl{\frac{t - 1}{2} }}, & t=1, \ 3 \leq t \leq n - 1,
  \end{cases} \\
  &= \begin{cases}
    A^{n^2 - \rf{\frac{n + 3}{2}}}, & t = 2, \\
    A^{n^2 - \rf{\frac{n - 1}{2}} - t + \fl{\frac{t - 1}{2}}}, & t=1, \ 3 \leq t \leq n - 1.
  \end{cases}
\end{align*}

One can easily check that the expression is maximised at $t=1$. Therefore, for $1 \leq t \leq n - 1$ we have
\begin{equation} \label{eq: det Xij vanish} 
  \fC\fD\fE_{t} \ll A^{n^2 - \rf{\frac{n + 1}{2}}}.
\end{equation}

%

Hence, combining~\eqref{eq: det Xij nonvanish} and~\eqref{eq: det Xij vanish}, we have overall
\[
  \#\cD_n(\cA; d) = \fA\fB + \sum_{t = 2}^{n - 1} \fC\fD\fE_t  \ll A^{n^2 - \rf{\frac{n + 1}{2}}},
\]
concluding the proof.

\begin{rem}
We note that while the bound of Theorem~\ref{thm:det} for $d\ne 0$ is dominated by~\eqref{eq: det Xij nonvanish},
we can eliminate the other bottleneck coming from~\eqref{eq: det Xij vanish}, which corresponds to the case $t=1$,
by showing that this case is impossible. Since this general argument can be useful for other similar questions,
we present it in Appendix~\ref{app: VanishSubDet}, see Proposition~\ref{prop:zero-cofactors}.
\end{rem}

\subsection{Proofs of Theorems~\ref{thm:poly 2-by-2} and~\ref{thm:poly general}}

\subsubsection{The Case $n = 2$ (Theorem~\ref{thm:poly 2-by-2})}  For each matrix
  \[
    \bX = \begin{bmatrix}
      x_{1,1} & x_{1,2} \\
      x_{2,1} & x_{2,2}
    \end{bmatrix} \in \cP_2 (\cA; T^2 - tT + d),
  \]
  the entries are related to the coefficients of the characteristic polynomial by the equations
  \begin{equation} \label{eq:poly:2-by-2:det-formula}
    x_{1,1}x_{2,2} - x_{1,2}x_{2,1} = d = \det\bX
  \end{equation}
  and
  \begin{equation} \label{eq:poly:2-by-2:trace-formula}
    x_{1,1} + x_{2,2} = t = \tr \bX.
  \end{equation}

  Suppose firstly that $t = 0$ and $d \neq 0$. Hence by~\eqref{eq:poly:2-by-2:trace-formula} we have
  $x_{1,1} = -x_{2,2}$, which, when substituted into~\eqref{eq:poly:2-by-2:det-formula} yields
  \[
    -(x_{1,1})^{2} - x_{1,2} x_{2,1} = d.
  \]

  For each possible value of $x_{1,2}$, we have an equation in $(x_{1,1})^{2}$ and $x_{2,1}$ with $O(1)$ solutions
  by Lemma~\ref{lem:LinEq inhomog}. This induces at most two values for $x_{1,1}$ and then for each of these, a
  unique value of $x_{2,2}$ by~\eqref{eq:poly:2-by-2:trace-formula}. Hence up to a constant, the value of $x_{1,2}$
  determines the rest of the matrix, so there are only $O(A)$ such matrices.

  Now, suppose that $t \neq 0$ and $d = 0$. The equation~\eqref{eq:poly:2-by-2:trace-formula} has $O(1)$ solutions
  in $x_{1,1}$ and $x_{2,2}$ ways by Lemma~\ref{lem:LinEq inhomog}.  Assuming these values are now fixed, we
  trivially have $O(A)$ solutions to~\eqref{eq:poly:2-by-2:det-formula} in $x_{1,2}$ and $x_{2,1}$ because either
  value uniquely determines the other. Thus there exist $O(A)$ such matrices.

  Finally, suppose $t \neq 0$ and $d \neq 0$. As above, we can solve for $x_{1,1}$ and $x_{2,2}$
  in~\eqref{eq:poly:2-by-2:trace-formula} in $O(1)$ ways by Lemma~\ref{lem:LinEq inhomog}. This allows us to
  similarly solve for $x_{1,2}$ and $x_{2,1}$ in~\eqref{eq:poly:2-by-2:det-formula} in $O(1)$ ways, again by
  Lemma~\ref{lem:LinEq inhomog}.

\subsubsection{The Case $n \geq 3$ (Theorem~\ref{thm:poly general})}

We construct an upper bound on $\#\cP_n(\cA; f)$ by acquiring an upper bound on the larger set of matrices
$\bX = (x_{i,j})_{i, j = 1}^n \in \cM_n(\cA)$ for which only the coefficients $c_{n - 1}$ and $c_{n - 2}$ of the
characteristic polynomial 
\[
  f = \det(T I_n - \bX) = \sum_{k = 0}^n c_k T^k,
\]
are fixed. 

Given that $c_{n - 1}$ and $c_{n - 2}$ are given by
\begin{equation} \label{eq: coeff12:second-coeff}
  c_{n - 1} = - \tr\bX \mand 
  c_{n - 2} = \frac{1}{2}((\tr\bX)^2 - \tr\bX^2),
\end{equation}
we may instead equivalently fix $t_1 = \tr \bX$ and $t_2 = \tr \bX^2$.

Fixing $t_1$ leads to an equation
\begin{equation} \label{eq:poly:general:trace}
  t_1 = \sum_{i = 1}^n x_{i,i},
\end{equation}
while fixing $t_2$, we have
\[
  t_2 = \sum_{i = 1}^n \sum_{j = 1}^n x_{i,j}x_{j,i} = \sum_{i = 1}^n x_{i,i}^2 + 2\sum_{1 \leq i < j \leq n}x_{i,j}x_{j,i},
\]
which leads to the equation
\begin{equation} \label{eq:poly:general:trace-squared}
  \frac{1}{2}\bigg(t_2 - \sum_{i = 1}^n x_{i,i}^2\bigg) = \sum_{1 \leq i < j \leq n}x_{i,j}x_{j,i}.
\end{equation}

We first note that there are at most
\[
  \fA \ll A^{\frac{n(n - 1)}{2}}
\]
possibilities for the elements $x_{i, j}$ for $1 \leq i < j \leq n$.

We begin by first counting the set of matrices for which
\begin{equation} \label{eq:trace-squared-condition}
  t_2 = \sum_{i = 1}^n x_{i, i}^2,
\end{equation}
and within this consider two cases,
\[
(t_1, t_2)  = (0,0) \mand (t_1, t_2)  \ne (0,0) .
\] 

If $(t_1, t_2) = (0, 0)$, then the number of possible values for the main diagonal is
$A^{\fl{\frac{2n}{5}}}$
by Lemma~\ref{lem:SystEq}. If either $t_1 \neq 0$ or $t_2 \neq 0$, then we may solve~\eqref{eq:poly:general:trace}
or~\eqref{eq:trace-squared-condition} respectively for all the $x_{i, i}$ in
$A^{\fl{\frac{n - 1}{2}}}$
ways by Lemma~\ref{lem:LinEq inhomog}, where in the second case we count the solutions over the smaller set $\{
x^2 :~x \in A\} \subseteq \Gamma$ which then fixes each $x_{i, i}$ up to a constant. This means that overall the
number of possibilities for the main diagonal is given by
\[
  \fB \ll \begin{cases}
    A^{\fl{\frac{2n}{5}}},
      &\quad (t_1, t_2) = (0, 0), \\
    A^{\fl{\frac{n - 1}{2}}},
      &\quad \text{otherwise}. \\
  \end{cases}
\]

From~\eqref{eq:trace-squared-condition}, the left hand side of~\eqref{eq:poly:general:trace-squared} is zero and
hence the number of possibilities for $x_{i, j}$ with $1 \leq j < i \leq n$ is
\[
  \fC \ll
    A^{\fl{\frac{n(n - 1)}{4}}}
\]
by Lemma~\ref{lem:LinEq homog}.

This means that the set of matrices with a fixed trace and trace squared which
satisfy~\eqref{eq:trace-squared-condition} is given by
\begin{equation} \label{eq:poly-first-bound}
  \fA \fB \fC \ll \begin{cases}
    A^{\fl{\frac{2n}{5}} + \frac{n(n - 1)}{2} + \fl{\frac{n(n - 1)}{4}}},
      \quad & (t_1, t_2) = (0, 0), \\
    A^{\fl{\frac{n - 1}{2}} + \frac{n(n - 1)}{2} + \fl{\frac{n(n - 1)}{4}}},
      \quad & \text{otherwise}. \\
  \end{cases}
\end{equation}

Now we consider the complementary case to~\eqref{eq:trace-squared-condition} of matrices for which
\begin{equation} \label{eq:trace-squared-condition-negation}
  t_2 \neq \sum_{i = 1}^n x_{i, i}^2.
\end{equation}

By considering solutions to~\eqref{eq:poly:general:trace}, we can see that number of possibilities for $x_{1,1},
\ldots, x_{n,n}$ is
\[
  \fD \ll \begin{cases}
    A^{\fl{\frac{n}{2}}},\quad & t_1 = 0, \\
    A^{\fl{\frac{n - 1}{2}}},\quad & t_1 \neq 0,
  \end{cases}
\]
by Lemma~\ref{lem:LinEq homog} and Lemma~\ref{lem:LinEq inhomog} respectively.

Now, in this case, from~\eqref{eq:trace-squared-condition-negation}, the left hand side
of~\eqref{eq:poly:general:trace-squared} is non-zero and hence the number of possibilities for $x_{i, j}$ with $1
\leq j < i \leq n$ is
\[
  \fE \ll 
  A^{\fl{\frac{n(n - 1)}{4} - \frac{1}{2}}}
\]
by Lemma~\ref{lem:LinEq inhomog}.

Thus the number of matrices of a fixed trace and trace squared satisfying~\eqref{eq:trace-squared-condition-negation} is
\begin{equation} \label{eq:poly-second-bound}
  \fA \fD \fE \ll \begin{cases}
    A^{\fl{\frac{n}{2}} + \frac{n(n - 1)}{2} + \fl{\frac{n(n - 1)}{4} - \frac{1}{2}}},
      \quad & t_1 = 0, \\
    A^{\fl{\frac{n - 1}{2}} + \frac{n(n - 1)}{2} + \fl{\frac{n(n - 1)}{4} - \frac{1}{2}}},
      \quad & t_1 \neq 0. \\
  \end{cases}
\end{equation}

Clearly, the second bound in~\eqref{eq:poly-second-bound} is always dominated by the first bound
in~\eqref{eq:poly-second-bound}. One also checks that for $n\ge 3$ we have $\fl{2n/5}  \le\fl{(n-1)/2}$.
Hence, the first bound in~\eqref{eq:poly-first-bound} is always dominated by the second one.

Therefore, the bounds~\eqref{eq:poly-first-bound} and~\eqref{eq:poly-second-bound} imply 
\[
  \cP_n (\cA; f) \ll   A^{\alpha(n)},
\]
where  $\alpha(n)$ is given by~\eqref{eq:alpha_n}. 

\section{Further questions}

There are various possible generalisations and extensions of the problems we consider here.  Firstly, it is quite
natural to consider special types of matrices such as those with symmetry constraints including symmetric, skew
symmetric, or Hermitian matrices, as considered in~\cite{Ful, EskKat, DRS}. We expect these questions require new
ideas. For instance, the method of fixing the trace and trace of the square as in Theorem~\ref{thm:poly general}
lends itself particularly poorly to counting symmetric matrices.

Motivated by recent work on counting commuting pairs of matrices~\cite{BSW, Mud}, one can also ask about an upper
bound on the number of commuting pairs $\bX \bY = \bY \bX $ with $\bX, \bY\in  \cM_n (\cA)$. Once can also ask
about multiplicative dependencies in $s$-tuples of matrices from $ \cM_n (\cA)$, similarly to questions studied
in~\cite{BOS,HOS}. For example, Theorem~\ref{thm:det} combined with Theorems~\ref{thm:poly 2-by-2}
and~\ref{thm:poly general} enables us to apply some ideas from~\cite{HOS} for such questions (at least for sets
$\cA\subseteq \Z$).

Since the work of Blomer and Li~\cite{BlLi} is a part of our motivation, it is  natural to investigate the same
type of applications as in~\cite{BlLi} and thus study the statistics of gaps between values of linear forms in
elements of finite rank multiplicative subgroups of $\R^*$. One can also  generalise the results here to  matrices
with polynomials entries, evaluated on elements of $\cA \subseteq \Gamma$, similarly to~\cite{BlLi, MOS}.

Finally, one can ask about similar questions over fields of positive characteristic, for example for subsets of
finitely generated multiplicative group in the field of rational functions over a finite field.


\section*{Acknowledgement}

This work was supported, in part, by the Australian Research Council Grants DP230100530 and DP230100534. A.O. also
gratefully acknowledges the hospitality and support of the Max Planck Institute for Mathematics and Institut des
Hautes \'Etudes Scientifiques, where parts of her work have been carried out.

\appendix

\section{Vanishing minors in Laplace expansion}
\label{app: VanishSubDet} 

\begin{prop} \label{prop:zero-cofactors}
  Let $\K$ be a field, and suppose $n \geq 2$. For any non-singular matrix $\bX \in \cM_n (\K)$ with non-zero
  entries, the Laplace expansion about any row or column has at most $n - 2$ zero minors.
\end{prop}

We also remark that the assumption that $\bX$ has non-zero entries is stronger than necessary. It is sufficient to
require that the span of the rows or columns not including the one being expanded about contains a vector with no
zero entries.

\begin{proof} Without loss of generality, we can consider expansion about the top row. 
  Let $\bX = (x_{i, j})_{i,j = 1}^n \in \cM_n (\K)$, and similarly define
  \[
    \widetilde \bX  = \begin{bmatrix}
      x_{n,1} & x_{n,2} & \ldots & x_{n,n} \\
      x_{2,1}     & x_{2,2}     & \ldots & x_{2,n} \\
       \vdots     & \vdots      & \ddots & \vdots \\
      x_{n,1} & x_{n,2} & \ldots & x_{n,n} \\
    \end{bmatrix} \in \cM_n(\K)
  \]
  as the matrix obtained by replacing the first row of $\bX$ with the bottom row. Clearly $\widetilde \bX$ is
  singular, and hence the Laplace expansion about the first row yields
  \begin{equation} \label{eq:zero-cofactors-laplace}
    0 = \det \widetilde \bX = \sum_{j = 1}^n (- 1)^{j + 1} x_{n,j} \det\bX_{1,j},
  \end{equation} 
  where, as before, $\bX_{1, j}$ is the submatrix of $\bX$ obtained by removing the first row and $j$-th column
  of $\bX$, or equivalently $\widetilde \bX$.
  
  It is impossible $ \det\bX_{1,j}$ to be zero for all $j \in \{1, \ldots, n\}$, otherwise $\det \bX = 0$,
  contradicting the assumed non-singularity of $\bX$. It is likewise impossible for exactly $n - 1$ of the cofactors
  $\det \bX_j$ to be zero, otherwise, because in this case, since all the $x_{n,j}$ are non-zero, the right hand
  side of~\eqref{eq:zero-cofactors-laplace} is also  non-zero.

  Therefore, at most $n - 2$ of the cofactors may be zero.
\end{proof}

\section{Tighter bounds from the proof of Theorem~\ref{thm:poly general}}
\label{app: bound poly general}

By careful consideration of the bounds in~\eqref{eq:poly-first-bound} and~\eqref{eq:poly-second-bound} over distinct
cases based on the value of $t_1$, $t_2$, and $n$, determining the particular maximum in each instance, one can
prove a tigher bound on $\#\cP_n(\cA; f)$ than that of Theorem~\ref{thm:poly general}. We present this bound below,
without proof.

It is convenient to define 
\begin{equation} \label{eq:beta}
\beta(n) = \frac{3}{4} n^2 - \frac{1}{4} n.
\end{equation}

\begin{thm} \label{thm:bound poly general}
  Let $\K$ be a field of characteristic zero, and $\Gamma$ a rank $\rk $ subgroup of $\K^\ast$. For
  any monic polynomial $f = \sum_{k = 0}^n c_k T^k \in \K[T]$ of degree $n \geq 3$, and $\cA
  \subseteq \Gamma$ of finite cardinality $A$, we have
  \[
    \#\cP_n(\cA; f) \ll A^{\beta(n)} \cdot \begin{cases}
      A^{-\lambda(n)}, & \quad c_{n - 1} = c_{n - 2} = 0,  \\
      A^{-\mu(n)}, & \quad c_{n - 1} = 0 \ \text{and\/}\ c_{n - 2} \neq 0, \\ 
      A^{-\nu(n)}, & \quad c_{n - 1} \neq 0, \\
    \end{cases}
  \]
  where $\beta(n)$ is given by~\eqref{eq:beta} and furthermore 
  \begin{align*}
     & \lambda(n) =  
     \begin{cases}
        1/2, & \quad n = 5, \\
        1, & \quad n \equiv 0 \pmod 4, \\
        3/2, & \quad n \equiv 1 \pmod 4 \ \text{and\/}\   n \neq 5, \\
        1/2, & \quad n \equiv 2 \pmod 4, \\
        1, & \quad n \equiv 3 \pmod 4, 
      \end{cases}\\
    & \mu(n) = 
     \begin{cases}
      1, & \quad n \equiv 0, 3 \pmod 4, \\
      1/2, & \quad n \equiv 1, 2 \pmod 4,
    \end{cases}
  \end{align*}
  and 
  \[\nu(n) = 
   \begin{cases}
      1, & \quad n \equiv 0 \pmod 4, \\
      1/2, & \quad n \equiv 1 \pmod 4, \\
      3/2, & \quad n \equiv 2 \pmod 4, \\
      1, & \quad n \equiv 3 \pmod 4. 
    \end{cases}
  \]
\end{thm}

Furthermore, in the case when $\K = \R$ we can further improve the bound in a few particular cases by noticing that
when $t_2 = 0$, the left hand side of~\eqref{eq:poly:general:trace-squared} must be non-zero, allowing the use of
Lemma~\ref{lem:LinEq inhomog} rather than possibly Lemma~\ref{lem:LinEq homog}. This leads to the following result.

\begin{thm} \label{thm:bound poly real}
  Suppose $\Gamma$ is a rank $\rk$ subgroup of $\R^\ast$ and $n \geq 3$ with $n \equiv 0, 1 \pmod 4$. For any monic
  polynomial $f = \sum_{k = 0}^n c_k T^k \in \R[T]$ of degree $n$ and $\cA \subseteq \Gamma$ of finite cardinality
  $A$, we have
  \begin{itemize}
    \item if $c_{n - 1} \neq 0$ and $2 c_{n - 2} = c_{n - 1}$,
    \[
      \# \cP_n (\cA; f) \ll  A^{\beta(n)} \cdot \begin{cases}
        A^{ - 2}, & \quad n \equiv 0 \pmod 4, \\
        A^{ - 3 / 2}, & \quad n \equiv 1 \pmod 4, 
      \end{cases}
    \]
    where  $\beta(n)$ is given by~\eqref{eq:beta};
    \item if $c_{n - 1} = c_{n - 2} = 0$ and $n = 5$,
    \[
      \# \cP_n (\cA; f) \ll A^{\beta(5) - 3/2}.
    \]
  \end{itemize}
\end{thm}

Note that for $n \equiv 2,3  \pmod 4$, the fact that $\Gamma \subseteq\R^\ast$ does not offer any 
advantage. 






\end{document}